\pdfoutput=1
\documentclass[10pt,reqno]{amsart}
\usepackage{enumerate,amssymb,amsthm}
\usepackage[hypertexnames=false]{hyperref}
\usepackage{amsmath,amscd}
\usepackage{mathtools}
\usepackage[nospace,noadjust]{cite}
\usepackage[margin=1in]{geometry}  
\usepackage{array}
\usepackage{tikz}
\usetikzlibrary{decorations.markings,arrows,arrows.meta,cd,calc,decorations.pathmorphing,decorations.pathreplacing,backgrounds,
positioning,fit,snakes}
\tikzset{snake it/.style={decorate, decoration={snake, segment length=3mm, amplitude=0.5mm}}}

\newtheorem{thm}{Theorem}
\newtheorem{cor}[thm]{Corollary}
\newtheorem{lem}[thm]{Lemma}
\newtheorem{prop}[thm]{Proposition}
\theoremstyle{definition}
\newtheorem{rem}[thm]{Remark}
\newtheorem*{que1}{Question 1}

\newcommand{\PMod}{\operatorname{PMod}}
\newcommand{\Z}{\mathbb Z}

\renewcommand{\S}{S}
\newcommand{\Push}{p}
\renewcommand{\Cap}{\operatorname{Cap}}
\newcommand{\Homeo}{\operatorname{Homeo}}

\renewcommand{\strut}{\rule[-5pt]{0pt}{15pt}}

\begin{document} 

\date{\today}

\title{Linearity of some low-complexity mapping class groups}

\author{Ignat Soroko}
\address{Department of Mathematics\\
Louisiana State University\\
Baton Rouge\\ LA 70803\\ USA}
\email{ignatsoroko@lsu.edu}


\begin{abstract}
By analyzing known presentations of the pure mapping groups of orientable surfaces of genus $g$ with $b$ boundary components and $n$ punctures, we show that these groups are isomorphic to some groups related to the braid groups and the Artin group of type $D_4$ in the cases when $g=0$ with $b$ and $n$ arbitrary, and when $g=1$ and $b+n$ is at most $3$. As a corollary, we conclude that the pure mapping class groups are linear in these cases.
\end{abstract}

\maketitle	

\section{Introduction}
Let $\S_{g,b,n}$ denote the orientable surface of genus $g$ with $b$ boundary components and $n$ punctures. Let $\PMod_{g,b,n}$ denote the pure mapping class group of orientation-preserving diffeomorphisms of $\S_{g,b,n}$ identical on the boundary and not permuting punctures, up to the isotopies identical on the boundary and not permuting punctures.

Presentations of $\PMod_{g,b,n}$ are known. Gervais~\cite{Ger} discovered a symmetrical presentation for $\PMod_{g,n,0}$ for $g>0$ in terms of commutator, braid, and star relations only. Labru\`ere and Paris have obtained in~\cite{LP} the presentations for $\PMod_{g,b,n}$ with arbitrary $g>0$, $b$ and $n$ in terms of quotients of Artin groups. However, both Gervais and Labru\`ere--Paris did not explicitly lay out to which known groups their presentation simplifies to when genus is equal to $1$ and the number of punctures and boundary components is small. As a result, a series of remarkable isomorphisms between such mapping class groups and some groups closely related to the braid group $B_4$ and to the Artin group of type $D_4$ have remained unknown to the general mathematical community. The purpose of this note is to make these beautiful connections to be broadly known. 

We also provide a simple proof for the presentation of $\PMod_{0,m,n}$, $(m>1)$, the pure mapping class group of the sphere with an arbitrary number of boundary components and punctures, which also seems to be missing in the literature.

Let $B_n$ denote the braid group on $n$ strands, $PB_n$ the pure braid group on $n$ strands, $A(D_4)$ the Artin group of type $D_4$, and $Z(G)$ the center of a group $G$. We establish the following theorem.

\begin{thm}\label{th5}
The following table lists the isomorphism types of groups $\PMod_{g,b,n}$ for the given triples $(g,b,n)$:
\begin{center}
{\renewcommand{\arraystretch}{0}%
\begin{tabular}{|c|c|}
\hline\strut
$(g,b,n)$ & $\PMod_{g,b,n}$\\
\hline \rule{0pt}{1pt}&\\
\hline\strut $(0,m,n)$, $m>1$ & $PB_{n+m-1}\times \Z^{m-1}$\strut\\
\hline\strut $(1,2,0)$ & $B_4 \times \Z$\\
\hline\strut $(1,1,1)$ & $B_4$ \\
\hline\strut $(1,0,2)$ & $B_4/Z(B_4)$\\
\hline\strut $(1,3,0)$ & $A(D_4)\times\Z^2$\strut\\
\hline\strut $(1,2,1)$ & $A(D_4)\times\Z$\\
\hline\strut $(1,1,2)$ & $A(D_4)$\\
\hline\strut $(1,0,3)$ & $A(D_4)/Z(A(D_4))$\\
\hline
\end{tabular}}\end{center}
\end{thm}

The genus $0$ case of Theorem~\ref{th5} is proved in Section~\ref{sect:0} (Proposition~\ref{prop:3}) and all other cases in Section~\ref{sect:1} (Propositions~\ref{prop8}, \ref{prop10}, Corollaries~\ref{cor9}, \ref{cor11}).

As a useful application we mention a well-known open problem to find out for which values of $g$, $b$, $n$ the group $\PMod_{g,b,n}$ admits a faithful linear representation. In her collection of open problems~\cite{Bir}, Joan Birman mentions that for the following triples $(g,b,n)$: 
\[
(1,0,0), (1,1,0), (1,0,1), (0,1,n), (0,0,n), (2,0,0)
\]
the mapping class group $\PMod_{g,b,n}$ is known to be linear, and she asks if this list can be extended to contain any other triples. We extend this list by the triples from Theorem~\ref{th5}:

\begin{cor}
The pure mapping class group $\PMod_{g,b,n}$ is linear for any of the triples $(g,b,n)$ from Theorem~\ref{th5}:
\[
(0,m,n),m>1,\quad (1,2,0), (1,1,1),(1,0,2),(1,3,0),(1,2,1),(1,1,2),(1,0,3).
\]
\end{cor}
\begin{proof}
Braid groups (and hence pure braid groups) are linear by the results of Krammer~\cite{Kra} and Bigelow~\cite{Big}. That Artin groups of spherical type are linear (in particular, $A(D_4)$), was proved by Cohen and Wales~\cite{CW} and independently by Digne~\cite{Dig}. The fact that if a group $G$ is linear then the quotient by its center $G/Z(G)$ is linear, follows from Theorem~6.4 in~\cite{Wehr}. And, of course, if two groups are linear, their direct product is linear as well.
\end{proof}

As another application, in Section~\ref{sec:HT} we answer in the negative a question of Hamidi-Tehrani (\cite{HT}) whether a group generated by certain positive multi-twists in $S_{1,2,0}$ is free of rank $2$.

\subsection*{Acknowledgments}
The author is grateful to 
Nikolai Ivanov and Dan Margalit for their useful comments on an early version of this paper.

\section{Genus \texorpdfstring{$0$}{0} case: \texorpdfstring{$(g,b,n)=(0,m,n)$, $m>1$}{(g,b,n)=(0,m,n)}}\label{sect:0}
In this section we obtain the following description of the mapping class group $\PMod_{0,m,n}$:
\begin{prop}\label{prop:3}
For $m>1$, $n\geq 0$,
\[
\PMod_{0,m,n}\cong \Z^{m-1}\times PB_{m+n-1}.
\]
If, in addition, $m+n\ge 3$, then 
\[
\PMod_{0,m,n}\cong \Z^m\times PB_{m+n-1}/Z(PB_{m+n-1}).
\]
\end{prop}

This result is mentioned in the literature (see e.g.~\cite[Lemma~3.4]{Har}), but we were unable to find a detailed proof of it anywhere, so we provide it here.

For completeness, we recall a few known results on $\PMod_{0,m,n}$, see~\cite[9.3, 4.2.4, Prop.\,2.4]{FM}:
\begin{gather}
\PMod_{0,0,n}\cong PB_{n-1}/Z(PB_{n-1}),\qquad \PMod_{0,1,n}\cong PB_n\cong PB_{n}/Z(PB_{n})\times\Z\quad(\text{if }n\ge2),\label{eq:1}\\
\PMod_{0,0,0}=\PMod_{0,0,1}=\PMod_{0,0,2}=\PMod_{0,0,3}=\PMod_{0,1,0}=\PMod_{0,1,1}=1,\label{eq:2}\\
\PMod_{0,0,4}=F_2,\qquad \PMod_{0,1,2}=\Z,\qquad \PMod_{0,1,3}=F_2\times \Z,\qquad \PMod_{0,2,0}=\Z.\label{eq:3}
\end{gather}

We will make use of the following lemma:
\begin{lem}[Capping the boundary,~{\cite[Prop.\,3.19]{FM}}]\label{lem3}
Let $S'$ be the surface obtained from a surface $S$ by capping one boundary component $\beta$ with a once-punctured disk. Let\/ $\Cap\colon\PMod(S)\to\PMod(S')$ be the induced homomorphism obtained by extending homeomorphisms of $S$ to the once-punctured disk by the identity. Then the following sequence is exact:
\[
1\longrightarrow\Z\longrightarrow\PMod(S)\xrightarrow{\,\,\,\Cap\,\,\,}\PMod(S')\longrightarrow1
\]
where $\Z$ is generated by the twist $T_\beta$ around $\beta$.\qed
\end{lem}

\begin{proof}[Proof of Proposition~\ref{prop:3}]
The second statement follows from the first one and~\eqref{eq:1} above. To prove the first statement, we argue by induction. Again, by~\eqref{eq:1}, $\PMod_{0,1,n}\cong PB_n\cong \Z^{1-1}\times PB_{1+n-1}$, which gives us the case of $m=1$ (and arbitrary $n\ge1$). Also, we know that $\PMod_{0,2,0}=\Z=\Z^{2-1}\times PB_1$. Suppose that the statement is true for $\PMod_{0,m,n+1}$ for some $m\ge1, n\ge0$, and we want to prove it for $\PMod_{0,m+1,n}$. Let $\delta_1$, \dots, $\delta_{m+1}$ denote the boundary circles of $\S_{0,m+1,n}$. By Lemma~\ref{lem3}, we have a short exact sequence 
\[
1\longrightarrow\Z\longrightarrow\PMod_{0,m+1,n}\xrightarrow{\,\,\,\Cap\,\,\,}\PMod_{0,m,n+1}\longrightarrow1,
\]
where the kernel $\Z$ is generated by the twist $T_{\delta_{m+1}}$ around $\delta_{m+1}$. It is sufficient to prove that this sequence splits as a direct product. For that we exhibit an epimorphism $f\colon\PMod_{0,m+1,n}\longrightarrow\Z$ such that the composition $\Z\longrightarrow\PMod_{0,m+1,n}\longrightarrow\Z$ is identical on $\Z$. Recall that the forgetful homomorphism ${\PMod(S\setminus\{x\})}\longrightarrow\PMod(S)$, which treats a mapping class of a surface $S$ without a point $x$ as a mapping class of $S$ is surjective (\cite[4.2.1]{FM}). We construct $f\colon\PMod_{0,m+1,n}\longrightarrow\PMod_{0,2,0}\cong \Z$ as a composition of $m-1$ capping epimorphisms, followed by $m-1+n$ forgetful epimorphisms, as shown in Figure~\ref{fig:1}.
\begin{figure}[ht!]
\begin{center}
\begin{tikzpicture}[thick]
\draw (0,0) ellipse (0.2cm and 1cm);
\draw (0,1)--(4,1);
\draw (0,-1)--(4,-1);
\draw[dashed] (4,1) arc (90:270:0.2cm and 1cm);
\draw (4,1) arc (90:-90:0.2cm and 1cm);

\draw (1,0.65) ellipse (0.3cm and 0.1cm);
\draw (2.5,0.65) ellipse (0.3cm and 0.1cm);
\draw (1.8,0.65) node {$\cdots$};
\draw (0,1.3) node {$\delta_1$};
\draw (1,1.3) node {$\delta_2$};
\draw (2.5,1.3) node {$\delta_m$};
\draw (4.25,1.3) node {$\delta_{m+1}$};

\fill (0.8,0) circle (0.05);
\fill (1.2,0) circle (0.05);
\fill (1.6,0) circle (0.05);
\fill (2.8,0) circle (0.05);
\draw (2.25,0) node {$\cdots$};
\draw [thin,snake=brace,segment amplitude=2mm] (2.9,-0.15) -- (0.7,-0.15); \draw (1.8,-0.55) node {\small$n$}; 

\begin{scope}[ultra thin]
\draw[dashed] (3.5,1) arc (90:270:0.2cm and 1cm);
\draw (3.5,1) arc (90:-90:0.2cm and 1cm);
\end{scope}

\draw[thin, ->] (4.75,0)--(6.25,0);
\draw[thin, ->] (7,-1.25) to [bend left] (3.85,-1.25);

\begin{scope}[xshift=7cm]
\draw (0,0) ellipse (0.2cm and 1cm);
\draw (0,1)--(3.5,1);
\draw (0,-1)--(3.5,-1);
\draw[dashed] (3.5,1) arc (90:270:0.2cm and 1cm);
\draw (3.5,1) arc (90:-90:0.2cm and 1cm);
\draw (0,1.3) node {$\delta_1$};
\draw (3.75,1.3) node {$\delta_{m+1}$};
\end{scope}
\end{tikzpicture}
\caption{\label{fig:1}The inclusions between surfaces $S_{0,m+1,n}$ and $S_{0,2,0}$ which induce the epimorphism $f\colon\PMod_{0,m+1,n}\longrightarrow\PMod_{0,2,0}$ and the identification $\PMod_{0,2,0}\cong\langle T_{\delta_{m+1}}\rangle$.}
\end{center}
\end{figure}
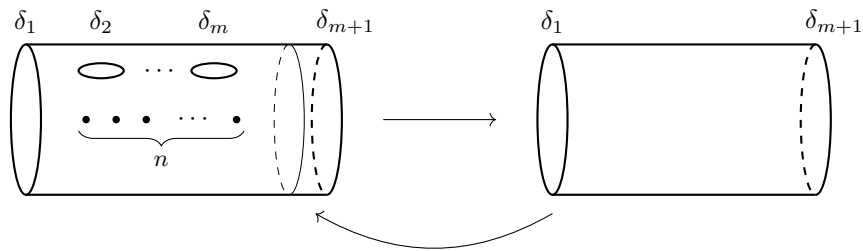
Obviously, the boundary twist $T_{\delta_{m+1}}$ maps under $f$ to the generator of $\PMod_{0,2,0}$ which can be identified with the boundary twist $T_{\delta_{m+1}}$ of $S_{0,m+1,n}$ itself. This shows that $\PMod_{0,m+1,n}\cong\Z\times\PMod_{0,m,n+1}$, and hence that $\PMod_{0,m+1,n}\cong\Z^m\times PB_{m+n}$ by induction.
\end{proof}

\section{Genus \texorpdfstring{$1$}{1} case}\label{sect:1}
In this section we analyze the Gervais presentation for $\PMod_{1,3,0}$ and show that it defines the group $A(D_4)\times \Z^2$. Then we apply Lemma~\ref{lem3} to obtain the descriptions of $\PMod_{1,2,1}$, $\PMod_{1,1,2}$, and $\PMod_{1,0,3}$ from Theorem~\ref{th5}. In the end, we do the same for $\PMod_{1,2,0}$, $\PMod_{1,1,1}$, and $\PMod_{1,0,2}$. 

\subsection*{The Gervais presentation for \texorpdfstring{$\PMod_{1,3,0}$}{PMod(1,3,0)}} In~\cite{Ger}, Gervais obtained a remarkable finite presentation for $\PMod(g,m,0)$ which is very symmetrical, though it is admittedly not the most economical in terms of the total number of generators and relations involved. For $\PMod_{1,3,0}$ it specifies to the following presentation, generated by the twists around the curves depicted in Figure~\ref{fig:2}. For simplicity, we denote a curve and its twist by the same letter.

\begin{figure}[ht!]
\begin{center}
\begin{tikzpicture}
\draw [very thick] (0,0) circle (0.75cm);
\draw [very thick] (2.934, 0.6237) arc (12:108:3cm);
\draw [very thick] (-2.007, 2.229) arc (132:228:3cm);
\draw [very thick] (-0.927, -2.853) arc (252:348:3cm);

\color{blue}
\draw [very thick] (2.934, 0.6237) arc (141.41:218.59:1cm);
\draw [very thick] (2.934, 0.6237) arc (110.674:249.326:0.67cm);  
\draw [very thick] (-2.007, 2.229) arc (261.41:338.59:1cm); 
\draw [very thick] (-2.007, 2.229) arc (230.674:369.326:0.67cm); 
\draw [very thick] (-0.927, -2.853) arc (21.41:98.59:1cm);
\draw [very thick] (-0.927, -2.853) arc (-9.326:129.326:0.67cm);
\draw (3.3,0) node {\Large$c_{31}$};
\draw (-1.5,2.75) node {\Large$c_{23}$};
\draw (-1.5,-2.75) node {\Large$c_{12}$};

\color{red}
\draw[thick] (-3,0) arc (180:0:1.125cm and 0.15cm);
\draw[dashed,thick] (-3,0) arc (180:360:1.125cm and 0.15cm);
\begin{scope}[rotate=-120]
\draw[thick] (-3,0) arc (180:0:1.125cm and 0.15cm);
\draw[thick,dashed] (-3,0) arc (180:360:1.125cm and 0.15cm);
\end{scope}
\begin{scope}[rotate=120]
\draw[thick,dashed] (-3,0) arc (180:0:1.125cm and 0.15cm);
\draw[thick] (-3,0) arc (180:360:1.125cm and 0.15cm);
\end{scope}
\draw (1.45,-1.6) node {\Large$a_1$};
\draw (1.45,1.6) node {\Large$a_3$};
\draw (-2,0.4) node {\Large$a_2$};

\draw [thick] (0,0) circle (1cm);
\draw (1.25,0) node {\Large$b$};

\color{blue}
\draw[thick,dashed] (0.927,2.853) arc (101.342:258.66:2.91cm);
\draw[thick] (0.927,2.853) arc (97.827:262.173:2.88cm);
\begin{scope}[rotate=120]
\draw[thick,dashed] (0.927,2.853) arc (101.342:258.66:2.91cm);
\draw[thick] (0.927,2.853) arc (97.827:262.173:2.88cm);
\end{scope}
\begin{scope}[rotate=-120]
\draw[thick,dashed] (0.927,2.853) arc (101.342:258.66:2.91cm);
\draw[thick] (0.927,2.853) arc (97.827:262.173:2.88cm);
\end{scope}
\draw (1.1,3.1) node {\Large$c_{13}$};
\draw (-3.3,-0.6) node {\Large$c_{32}$};
\draw (2.2,-2.5) node {\Large$c_{21}$};

\begin{scope}[xshift=8cm]
\color{black}
\fill (0,0) circle (2pt);
\fill (-1.5,0) circle (2pt);
\draw[very thick] (-1.5,0)--(0,0);
\begin{scope}[rotate=120]
\fill (0,0) circle (2pt);
\fill (-1.5,0) circle (2pt);
\draw[very thick] (-1.5,0)--(0,0);
\end{scope}
\begin{scope}[rotate=-120]
\fill (0,0) circle (2pt);
\fill (-1.5,0) circle (2pt);
\draw[very thick] (-1.5,0)--(0,0);
\end{scope}
\draw (0.4,0) node {\Large$b$};
\draw (1.2,-1.3) node {\Large$a_1$};
\draw (1.2,1.3) node {\Large$a_3$};
\draw (-1.5,-0.4) node {\Large$a_2$};

\end{scope}

\end{tikzpicture}
\caption{Gervais' generators of $\PMod_{1,3,0}$ and the Coxeter graph $D_4$.\label{fig:2}}
\end{center}
\end{figure}
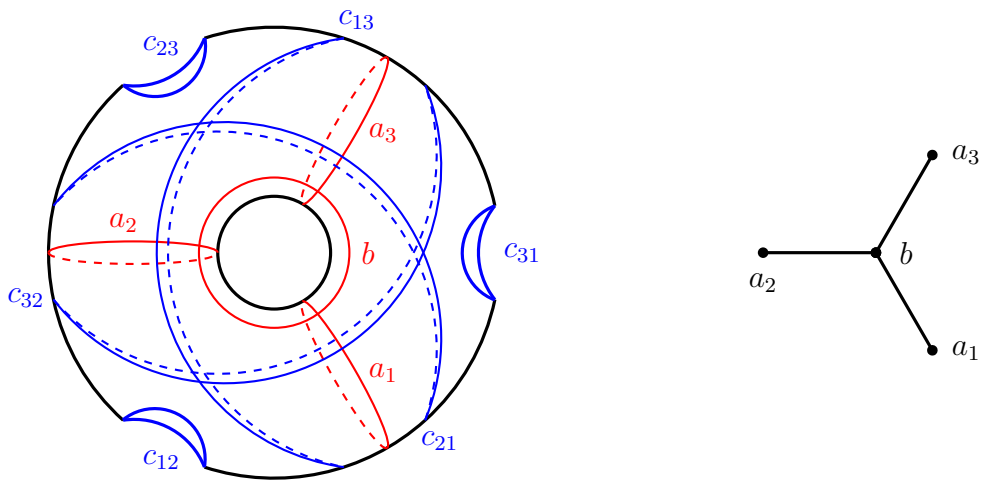

{\bf Generators:} $b$, $a_1$, $a_2$, $a_3$, $c_{12}$, $c_{21}$, $c_{13}$, $c_{31}$, $c_{23}$, $c_{32}$;

{\bf Relations:} 
\begin{itemize}
\item {\it Handles:} $\varnothing$ (present only when $g\ge2$);
\item {\it Commutators:} Any two of the above generators commute if and only if their curves are disjoint;
\item {\it Braids:} Any two of the above generators $x$, $y$ form a braid relation $xyx=yxy$ if and only if their curves intersect only once;
\item {\it Stars:} For any good triple $(i,j,k)\in\{1,2,3\}^3$, there is a star relation: 
\begin{equation}\label{eq:E}
(a_ia_ja_kb)^3=c_{ij}c_{jk}c_{ki},\tag{$E_{i,j,k}$}
\end{equation}
(with the assumption that $c_{\ell\ell}=1$);
\end{itemize}
A triple $(i,j,k)$ is called \emph{good}, if $(i,j,k)\ne(x,x,x)$ for all $x\in\{1,2,3\}$ and either $i\le j\le k$ or $j\le k\le i$ or $k\le i\le j$. 

\subsection*{Reducing the number of star relations} 
Notice that, when all $i,j,k$ are distinct, the element $\Delta_{ijk}:=(a_ia_ja_kb)^3$ in the left-hand side of the star relation~\eqref{eq:E} is the fundamental element of the Artin group of type $D_4$ on the generators $a_i, a_j, a_k, b$. In particular, $\Delta_{ijk}$ does not depend on the order of the generators inside the third power, see~\cite[{Satz~7.2 and Zusatz}]{BS}. Hence, $\Delta_{ijk}=\Delta_{jki}=\Delta_{kij}$ and the corresponding star relations $(E_{i,j,k})$, $(E_{j,k,i})$ and $(E_{k,i,j})$ coincide since $c_{ij}$, $c_{jk}$ and $c_{ki}$ pairwise commute. (Notice though, that this does not give us the equivalence of relations $(E_{i,j,k})$ corresponding to good triples with the ones corresponding to non-good triples, since the right-hand side of a non-good triple relation, say, $(E_{j,i,k})$, will be comprised of completely different curves $c_{ji}$, $c_{ik}$, $c_{kj}$.)

If in a good triple $(i,j,k)$ two elements are equal, say, $i=k$, then, 
by using only braid and commutation relations (underlined below), we obtain (cf. Lemma~2\,(iii) of~\cite{Ger}):
\begin{multline}\label{eq4}
\Delta_{iij}=(a_ia_ia_jb)^3 = a_i\,\underline{a_ia_j}\,b\cdot a_i\,\underline{a_ia_j}\,b\cdot a_ia_ia_jb= 
								a_ia_j\,\underline{a_iba_i}\,a_j\,\underline{a_iba_i}\,a_ia_jb=\\
							 a_ia_jba_i\,\underline{ba_jb}\,a_iba_ia_jb= 
							 a_ia_jb\cdot a_ia_jb\cdot \underline{a_ja_i}\,b\cdot a_ia_jb=(a_ia_jb)^4.
\end{multline}							
In particular, since $a_i$ and $a_j$ commute, we have $(a_ia_jb)^4=(a_ja_ib)^4$, and, reading the above equation backwards, we conclude that $\Delta_{iij}=\Delta_{ijj}$. We summarize our observations in the following remark.
\begin{rem}\label{rem5}
In the Gervais presentation, one only needs the good triples $(i,j,k)$ with $i\le j<k$.
\end{rem} 
(This was probably supposed to be the content of Remark~3 in~\cite{Ger}, but due to an unfortunate typo, this information was not properly conveyed to the reader.)
\begin{rem}
When all $i,j,k$ are distinct, $\Delta_{ijk}$ is the generator of the center of $A(D_4)$ (\cite{BS}). Similarly, $\Delta_{iij}$ is the generator of the center of $A(A_3)$. Indeed, as above, using the braid and commutation relations only, we get (as in the proof of Lemma~2 of~\cite{Ger}):
\begin{multline}\label{eq5}
(a_ia_jb)^4= a_ia_jb\cdot \underline{a_ia_j}\,b\cdot \underline{a_ia_j}\,b\cdot a_ia_jb=
a_i\underline{a_jba_j}\, a_iba_j\,\underline{a_iba_i}\,a_jb=\\
a_i b a_j\, \underline{ba_ib}\, a_jba_i\,\underline{ba_jb}=
a_iba_j\cdot a_i b\,\underline{a_ia_j}\,b\,\underline{a_ia_j}\,ba_j=(a_iba_j)^4,
\end{multline}
which is the expression for the generator of the center in the Artin group $A_3$ on the generators $a_i$, $a_j$, $b$, see~\cite{BS}. We will denote this element as $\Delta_{ij}^2$, since it is equal to the square of the fundamental element of the respective group.
\end{rem}

\begin{rem}
The Dehn twists involved in the Gervais presentation can be either all right or all left uniformly in all the relations. Indeed, right and left twists are inverses of each other, and braid and commutation relations are invariant under taking inverses. The same is true for the star relation, since its left-hand side doesn't depend on the order of the generators inside the third power (\cite{BS}), and its right-hand side consists of pairwise commuting twists. (One may argue that this should be true for arbitrary presentations of the mapping class group of an orientable surface in terms of Dehn twists as generators, since the definition of what is a left twist and what is a right twist around a curve depends on the way we look at the surface: `from the outside' or `from the inside'. Since homeomorphisms of the surface are defined intrinsically, the presentations of the mapping class groups should be invariant under changing all twists from the left ones to the right ones or vice versa.)
\end{rem}

\begin{prop}\label{prop8}
In the group\/ $\PMod_{1,3,0}$, the elements $a_1$, $a_2$, $a_3$, and $b$ generate a subgroup isomorphic to $A(D_4)$, and the elements $c_{12}$, $c_{23}$, $c_{31}$ generate a free abelian subgroup\/ $\Z^3$. The whole group $\PMod_{1,3,0}$ is isomorphic to their direct product modulo one star relation:
\[
\PMod_{1,3,0}\cong A(D_4)\times\Z^3\big/\big\langle c_{12}c_{23}c_{31}=(a_1a_2a_3b)^3\big\rangle.
\]
In particular, $\PMod_{1,3,0}\cong A(D_4)\times\Z^2$.
\end{prop}
\begin{proof}
We write down the Gervais presentation for $\PMod_{1,3,0}$, taking into account Remark~\ref{rem5}.

{\it Generators:}  $b$, $a_1$, $a_2$, $a_3$, $c_{12}$, $c_{21}$, $c_{13}$, $c_{31}$, $c_{23}$, $c_{32}$.

{\it Relations:} 
\begin{itemize}
\item Commutators (we write $x \rightleftarrows y$ for $xy=yx$): 

$a_i \rightleftarrows a_j$ for all $i,j$;

$a_1 \rightleftarrows c_{ij}$ for all $c_{ij}$, except $c_{32}$;

$a_2 \rightleftarrows c_{ij}$ for all $c_{ij}$, except $c_{13}$;

$a_3 \rightleftarrows c_{ij}$ for all $c_{ij}$, except $c_{21}$;

$b \rightleftarrows c_{ij}$ for all $i,j$;

$c_{12},c_{23},c_{31} \rightleftarrows c_{ij}$ for all $i,j$.

\item Braid relations: among $a_1,a_2,a_3,b$ only, as prescribed by the Coxeter graph $D_4$, see Figure~\ref{fig:2}.
\item Stars (only for triples $(i,j,k)$ with $i\le j<k$, using~\eqref{eq4}, \eqref{eq5}):

(1,1,2):\quad $c_{12}c_{21}=\Delta_{12}^2$, 

(1,1,3):\quad $c_{13}c_{31}=\Delta_{13}^2$, 

(1,2,3):\quad $c_{12}c_{23}c_{31}=\Delta_{123}$, 

(2,2,3):\quad $c_{23}c_{32}=\Delta_{23}^2$. 
\end{itemize}

To simplify this presentation, we eliminate three of $c_{ij}$ using the degenerate star relations above:
\[
c_{21}=c_{12}^{-1}\Delta_{12}^2, \qquad c_{13}=\Delta_{13}^2c_{31}^{-1}, \qquad c_{32}=c_{23}^{-1}\Delta_{23}^2.
\]
Let's show that with these eliminations, all the commutator and braid relations involving the eliminated letters are the consequences of the relations in the group $A(D_4)\times \Z^3(c_{12},c_{23},c_{31})$: 

$a_1\rightleftarrows c_{21}$: true since $a_1 \rightleftarrows c_{12}$ and $a_1 \rightleftarrows \Delta_{12}^2$ (the last element being central in $\langle a_1,b,a_2\rangle$). 

A similar reasoning shows that $a_1\rightleftarrows c_{13}$, $a_2\rightleftarrows c_{21}$, $a_2\rightleftarrows c_{32}$, $a_3\rightleftarrows c_{13}$, $a_3\rightleftarrows c_{32}$, and also that $b\rightleftarrows c_{ij}$ and $c_{12},c_{23},c_{31}\rightleftarrows c_{13},c_{32},c_{21}$.

This shows that $\PMod_{1,3,0}$ is isomorphic to $A(D_4)\times \Z^3$ modulo a single star relation for the triple $(1,2,3)$.

To show that $\PMod_{1,3,0}$ is isomorphic to $A(D_4)\times\Z^2$, we eliminate one of the generators of $\Z^3$, say, $c_{31}$ from the star relation:
\[
c_{31}=c_{23}^{-1}c_{12}^{-1}\Delta_{123},
\]
and observe (using the fact that $\Delta_{123}$ is central in $A(D_4)$) that all the commutation relations involving $c_{31}$ in $A(D_4)\times\Z^3$ still hold in $A(D_4)\times\Z^2(c_{12},c_{23})$.
\end{proof}

\begin{cor}\label{cor9}
We have the following isomorphisms:
\begin{itemize}
\item $\PMod_{1,2,1}\cong A(D_4)\times\Z;$
\item $\PMod_{1,1,2}\cong A(D_4);$
\item $\PMod_{1,0,3}\cong A(D_4)/Z(A(D_4))$.
\end{itemize}
\end{cor}
\begin{proof}
By applying Lemma~\ref{lem3} three times to the presentation
\[
\PMod_{1,3,0}\cong A(D_4)\times\Z^3(c_{12},c_{23},c_{31})\big/\big\langle c_{12}c_{23}c_{31}=(a_1a_2a_3b)^3\big\rangle,
\]
i.e. capping the boundary twists $c_{31}$, $c_{23}$ and $c_{12}$ with a punctured disk (in that order), and arguing as in the end of the proof of Proposition~\ref{prop8}, we get:
\begin{align*}
\PMod_{1,2,1}\cong {}& A(D_4)\times\Z^2(c_{12},c_{23})\big/\big\langle c_{12}c_{23}=(a_1a_2a_3b)^3\big\rangle\cong A(D_4)\times\Z;\\
\PMod_{1,1,2}\cong {}& A(D_4)\times\Z(c_{12})\big/\big\langle c_{12}=(a_1a_2a_3b)^3\big\rangle\cong A(D_4);\\
\PMod_{1,0,3}\cong {}& A(D_4)\big/\big\langle (a_1a_2a_3b)^3\big\rangle\cong A(D_4)/Z(A(D_4)),
\end{align*}
the last isomorphism taking place since $(a_1a_2a_3b)^3$ generates the center of $A(D_4)$, see~\cite[{Satz~7.2 and Zusatz}]{BS}.
\end{proof}

Now we will deal with the group $\PMod_{1,2,0}$ and its relatives in a similar, but simpler, way. Let generators $a_1$, $a_2$, $b$, $c_{12}$ and $c_{21}$ be as in the Figure~\ref{fig:3}. 

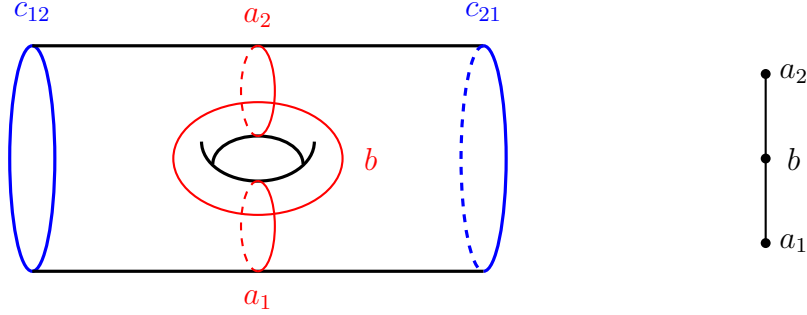
\begin{figure}[ht!]
\begin{center}
\begin{tikzpicture}[thick,scale=1.5]
\begin{scope}[very thick]
\draw [blue] (0,0) ellipse (0.2cm and 1cm);
\draw (0,1)--(4,1);
\draw (0,-1)--(4,-1);
\draw[blue,dashed] (4,1) arc (90:270:0.2cm and 1cm);
\draw[blue] (4,1) arc (90:-90:0.2cm and 1cm);
\draw[blue] (0,1.3) node {\Large$c_{12}$};
\draw[blue] (4,1.3) node {\Large$c_{21}$};

\draw (2.4,-0.05) arc (0:180:0.4cm and 0.25cm);
\draw (1.5,0.15) arc (180:360:0.5cm and 0.35cm);
\end{scope}

\color{red}
\draw (2,0) ellipse (0.75cm and 0.5cm);
\draw[dashed] (2,1) arc (90:270:0.15cm and 0.4cm);
\draw (2,1) arc (90:-90:0.15cm and 0.4cm);
\draw (2,-1) arc (-90:90:0.15cm and 0.4cm);
\draw[dashed] (2,-0.2) arc (90:270:0.15cm and 0.4cm);
\draw (2,-1.25) node {\Large$a_1$};
\draw (2,1.25) node {\Large$a_2$};
\draw (3,0) node {\Large$b$};

\begin{scope}[black,xshift=6.5cm]
\fill (0,0) circle (1.25pt);
\fill (0,0.75) circle (1.25pt);
\fill (0,-0.75) circle (1.25pt);
\draw (0,0.75)--(0,0)--(0,-0.75);
\draw (0.25,0.75) node {\Large$a_2$};
\draw (0.25,0) node {\Large$b$};
\draw (0.25,-0.75) node {\Large$a_1$};
\end{scope}

\end{tikzpicture}
\caption{Gervais' generators of $\PMod_{1,2,0}$ and the Coxeter graph $A_3$.\label{fig:3}}
\end{center}
\end{figure}

We have the following Proposition.

\begin{prop}\label{prop10}
In the group\/ $\PMod_{1,2,0}$, the elements $a_1$, $a_2$, and $b$ generate a subgroup isomorphic to $B_4$, and the elements $c_{12}$, $c_{21}$ generate a free abelian subgroup\/ $\Z^2$. The whole group $\PMod_{1,2,0}$ is isomorphic to their direct product modulo one (degenerate) star relation:
\[
\PMod_{1,2,0}\cong B_4\times\Z^2\big/\big\langle c_{12}c_{21}=(a_1ba_2)^4\big\rangle.
\]
In particular, $\PMod_{1,2,0}\cong B_4\times\Z$.
\end{prop}
\begin{proof}
The proof is analogous to the proof of Proposition~\ref{prop8}. The Gervais presentation for $\PMod_{1,2,0}$ is the following:

{\it Generators:}  $b$, $a_1$, $a_2$, $c_{12}$, $c_{21}$.

{\it Relations:} 
\begin{itemize}
\item Commutators: 

$a_1 \rightleftarrows a_2$; \qquad 
$a_1, a_2, b \rightleftarrows c_{12}, c_{21}$; \qquad
$c_{12} \rightleftarrows c_{21}$.

\item Braid relations: among $a_1,a_2,b$ only, as prescribed by the Coxeter graph $A_3$, see Figure~\ref{fig:3}.
\item Stars (only for triples $(i,j,k)$ with $i\le j<k$, using~\eqref{eq4}, \eqref{eq5}):

(1,1,2):\quad $c_{12}c_{21}=\Delta_{12}^2$, 
\end{itemize}
We see that the Gervais presentation gives us the presentation 
\[
\PMod_{1,2,0}\cong B_4\times\Z^2\big/\big\langle c_{12}c_{21}=(a_1ba_2)^4\big\rangle
\]
on the nose. To show that $\PMod_{1,2,0}\cong B_4\times\Z$, we eliminate $c_{21}$ from the star relation:
\[
c_{21}=c_{12}^{-1}\Delta_{12}^2,
\]
and check that all the commutation relations from $B_4\times\Z^2(c_{12},c_{21})$ involving $c_{21}$ hold true in $B_4\times\Z$. Again, as above, we are using the fact that $\Delta_{12}^2$ is central in $B_4$. 
\end{proof}

\begin{cor}\label{cor11}
We have the following isomorphisms:
\begin{itemize}
\item $\PMod_{1,1,1}\cong B_4;$
\item $\PMod_{1,0,2}\cong B_4/Z(B_4)$.
\end{itemize}
\end{cor}
\begin{proof}
Again, as above, we apply Lemma~\ref{lem3} two times to the presentation
\[
\PMod_{1,2,0}\cong B_4\times\Z^2(c_{12},c_{21})\big/\big\langle c_{12}c_{21}=(a_1ba_2)^4\big\rangle.
\]
which amounts to capping the boundary twists $c_{21}$, $c_{12}$ with punctured disks. We get:
\begin{align*}
\PMod_{1,1,1}\cong {}& B_4\times\Z(c_{12})\big/\big\langle c_{12}=(a_1ba_2)^4\big\rangle\cong B_4;\\
\PMod_{1,0,2}\cong {}& B_4\big/\big\langle (a_1ba_2)^4\big\rangle\cong B_4/Z(B_4),
\end{align*}
the last isomorphism taking place since $(a_1ba_2)^4$ generates the center of $B_4$.
\end{proof}

\begin{rem}
One can obtain the isomorphisms from Theorem~\ref{th5} ``from first principles'', starting with the presentation for $\PMod_{1,1,0}\cong B_3$ and reverting the process of capping a boundary component with a punctured disk followed by the forgetful map. The basic tools here are Lemma~\ref{lem3}, the Birman exact sequence~\cite[4.2]{FM} (which describes the kernel of the forgetful map), and the method of constructing a finite presentation of an extension of two finitely presented groups~\cite[Prop.~10.2.1]{John}. The author undertook such an approach in an early version of this paper, which resulted in a much longer text full of computations. An interested reader is welcome to request it from the author.
\end{rem}

\section{On a question of Hamidi-Tehrani}\label{sec:HT}
In~\cite{HT}, Hamidi-Tehrani studied subgroups of mapping class groups generated by positive multi-twists. In the last section he asks the following question:
\begin{que1}
If $a_1$, $a_2$, $b$ are the Dehn twists as in the Figure~\ref{fig:3}, is it true that the group $\langle a_1^2a_2, b\rangle$ is free of rank $2$?
\end{que1}

We answer this question in the negative. Indeed, we saw in~\eqref{eq4} that the element $\Delta_{112}=(a_1^2a_2\cdot b)^3=\Delta_{12}^2$ is central in $B_4$, and hence is central and nontrivial in the group in question. Hence, this group cannot be free non-abelian.


\begin{thebibliography}{99}
\bibitem[Big]{Big} Stephen J. Bigelow, Braid groups are linear. {\it J. Amer. Math. Soc.} 14 (2001), no. 2, 471--486.
\bibitem[Bir]{Bir} Joan Birman, The topology of 3--manifolds, Heegaard distance and the mapping class group of a 2--manifold. Problems on mapping class groups and related topics, 133--149, {\it Proc. Sympos. Pure Math.}, 74, Amer. Math. Soc., Providence, RI, 2006.
\bibitem[BS]{BS} Egbert Brieskorn, Kyoji Saito, Artin-Gruppen und Coxeter-Gruppen. {\it Invent. Math.} 17 (1972), 245--271. 
\bibitem[CW]{CW} Arjeh M. Cohen, David B. Wales, Linearity of Artin groups of finite type. {\it Israel J. Math.} 131 (2002), 101--123. 
\bibitem[Dig]{Dig} Fran\c{c}ois Digne, On the linearity of Artin braid groups. {\it J. Algebra}, 268 (2003), no. 1, 39--57.
\bibitem[FM]{FM} Benson Farb, Dan Margalit, A primer on mapping class groups. {\it Princeton Mathematical Series}, 49. Princeton University Press, Princeton, NJ, 2012. xiv+472 pp.
\bibitem[Ger]{Ger} Sylvain Gervais, A finite presentation of the mapping class group of a punctured surface. {\it Topology}, 40 (2001), no. 4, 703--725. 
\bibitem[HT]{HT} Hessam Hamidi-Tehrani, Groups generated by positive multi-twists and the fake lantern problem. {\it Algebr. Geom. Topol.} 2 (2002), 1155--1178.
\bibitem[Har]{Har} John Harer, The second homology group of the mapping class group of an orientable surface. {\it Invent. Math.} 72 (1983), no. 2, 221--239. 
\bibitem[John]{John} D.L. Johnson, Presentations of groups. Second edition. London Mathematical Society Student Texts, 15. {\it Cambridge University Press,} Cambridge, 1997. xii+216 pp.
\bibitem[Kra]{Kra} Daan Krammer, Braid groups are linear. {\it Ann. of Math.} (2) 155 (2002), no. 1, 131--156. 
\bibitem[LP]{LP} Catherine Labru\`ere, Luis Paris, Presentations for the punctured mapping class groups in terms of Artin groups. {\it Algebr. Geom. Topol.} 1 (2001), 73--114.
\bibitem[Wehr]{Wehr} Bertram A.F. Wehrfritz, Infinite Linear Groups. An account of the group-theoretic properties of infinite groups of matrices. {\it Ergebnisse der Matematik und ihrer Grenzgebiete}, Band 76. Springer--Verlag, New York--Heidelberg, 1973. xiv+229 pp.
\end{thebibliography}
\end{document}